\documentclass[runningheads]{llncs}
\usepackage[T1]{fontenc}
\usepackage{amssymb}
\usepackage{amsmath}
\usepackage{bm}

\usepackage{lscape}
\usepackage{graphics}
\usepackage{times}
\usepackage{color}
\usepackage{colortbl}
\usepackage{caption}
\usepackage{epsfig,latexsym,graphicx}
\usepackage{psfrag}
\usepackage{comment}

\usepackage{amsmath,amscd,amssymb,amsfonts}
\usepackage{xspace}
\usepackage{listings}
\usepackage{subfig}
\usepackage{boxedminipage}
\usepackage{url}
\usepackage[ruled,vlined,linesnumbered]{algorithm2e}
\usepackage{multirow}
\usepackage{rotating}
\usepackage{wrapfig}
\usepackage{xcolor, soul}
\sethlcolor{green}
\newtheorem{observation}{Observation}
\begin{document}
\title{Optimal $L(1,2)$-edge Labeling of Infinite Octagonal Grid}
\author{Subhasis Koley \and Sasthi C. Ghosh}
\institute{Advanced Computing and Microelectronics unit\\ Indian Statistical Institute, 203 B. T. Road, Kolkata 700108, India\\ Emails: \email{subhasis.koley2@gmail.com, sasthi@isical.ac.in }\\
}
\authorrunning{Sasthi C. Ghosh et al.}
\maketitle              

\begin{abstract} 
For two given non-negative integers $h$ and $k$, an $L(h,k)$-edge labeling of a graph $G=(V(G),E(G))$ is a function $f':E(G) \xrightarrow{}\{0,1,\cdots, n\}$ such that $\forall e_1,e_2 \in E(G)$, $\vert f'(e_1)-f'(e_2) \vert \geq h$ when $d'(e_1,e_2)=1$ and $\vert f'(e_1)-f'(e_2) \vert \geq k$ when $d'(e_1,e_2)=2$ where $d'(e_1,e_2)$ denotes the distance between $e_1$ and $e_2$ in $G$. Here $d'(e_1,e_2)=k'$ if there are at least $(k'-1)$ number of edges in $E(G)$ to connect $e_1$ and $e_2$ in $G$. The objective is to find \textit{span} which is the minimum $n$ over all such $L(h,k)$-edge labeling and is denoted as $\lambda'_{h,k}(G)$. Motivated by the channel assignment problem in wireless cellular network, $L(h,k)$-edge labeling problem has been studied in various infinite regular grids. For infinite regular octagonal grid $T_8$, it was proved  that $25 \leq \lambda'_{1,2}(T_8) \leq 28$ [Tiziana Calamoneri, International Journal of Foundations of Computer Science, Vol. 26, No. 04, 2015] with a gap between lower and upper bounds. In this paper we fill the gap and prove that $\lambda'_{1,2}(T_8)= 28$.

\end{abstract}

\section{Introduction}\label{sec:1}
In a wireless network, frequency channels are assigned to the transmitters for communications. As the number of frequency channels are limited, the main challenge of the Channel Assignment Problem ($CAP$) is to assign frequencies to the transmitters effectively such that interference can not occur. Hale~\cite{Hale} formulated $CAP$ as a classical vertex coloring problem. To incorporate the effect of interference in multi hop distance into the formulation, distance vertex labeling of graph was proposed in ~\cite{Roberts},~\cite{GriggsYeh}. A special case of distance vertex labeling of a graph is $L(h,k)$-vertex labeling and is defined as follows.

\begin{definition}
For two given non-negative integers $h$ and $k$, an $L(h,k)$-vertex labeling of a graph $G=(V(G),E(G))$ is a function $f:V(G) \xrightarrow{}\{0,1,\cdots, n\}$ such that $\forall v_1,v_2 \in V(G)$, $\vert f(v_1)-f(v_2) \vert \geq h$ when $d(v_1,v_2)=1$ and $\vert f(v_1)-f(v_2) \vert \geq k$ when $d(v_1,v_2)=2$ where $d(v_1,v_2)$ denotes the distance between $v_1$ and $v_2$. Here $d(v_1,v_2)=k'$ if there are at least $k'$ number of edges in $E(G)$ to connect $v_1$ and $v_2$ in $G$.
\end{definition}

The \textit{span} $\lambda_{h,k}(G)$ of $L(h,k)$-vertex labeling is the minimum $n$ such that $G$ admits an $L(h,k)$-vertex labeling. The $L(h,k)$-vertex labeling problem on different types of graphs have been studied  ~\cite{YEH},~\cite{cala2},~\cite{GriggsJin}. The $L(h,k)$-edge labeling problem was introduced in ~\cite{Georges} and is defined as follows.

\begin{definition}
For two given non-negative integers $h$ and $k$, an $L(h,k)$-edge labeling of a graph $G=(V(G),E(G))$ is a function $f':E(G) \xrightarrow{}\{0,1,\cdots, n\}$ such that $\forall e_1,e_2 \in E(G)$,  $\vert f'(e_1)-f'(e_2) \vert \geq h$ when $d'(e_1,e_2)=1$ and $\vert f'(e_1)-f'(e_2) \vert \geq k$ when $d'(e_1,e_2)=2$ where $d'(e_1,e_2)$ denotes the distance between $e_1$ and $e_2$. Here $d'(e_1,e_2)=k'$ if there are at least $(k'-1)$ number of edges in $E(G)$ to connect $e_1$ and $e_2$ in $G$. 
\end{definition}
The \textit{span} $\lambda'_{h,k}(G)$ of $L(h,k)$-edge labeling is the minimum $n$ such that $G$ admits an $L(h,k)$-edge labeling. Various studies for $L(h,k)$-edge labelling have been carried out for different types of graphs. Among the various graph classes, many studies have been done on infinite regular hexagonal grid, infinite square grid, infinite triangular grid and infinite octagonal grid specially for $h=1,2$ and $k=1,2$ ~\cite{lin1},~\cite{lin2},~\cite{lin3},~\cite{cala}~\cite{Bandopadhyay2021}. In~\cite{cala},  infinite octagonal grid $T_8$ has been studied and the lower and upper bound of $\lambda'_{1,2}(T_8)$ have been given. More specifically, in~\cite{cala}, it is given that $25 \leq \lambda'_{1,2}(T_8) \leq 28$, where there is a gap between lower and upper bounds. In this paper, we fill the gap and prove that $\lambda'_{1,2}(T_8)\geq 28$. As in~\cite{cala}, $\lambda'_{1,2}(T_8)\leq 28$, it is concluded that $\lambda'_{1,2}(T_8)= 28$. Throughout this paper, we use color and label interchangeably.  To prove the lower bound of $\lambda'_{1,2}(T_8)$, we use the structural properties of $T_8$. More specifically, we first identify a subgraph $G'$ of $T_8$ where no two edges can have the same color. After that we consider all  edges where a pair of consecutive colors $(c,c\pm 1)$ can be used in $G'$. Based on this, we identify the subgraphs in $T_8$ where $c$ and $c\pm 1$ can not be used. Then using the structural properties of those subgraphs we conclude how many additional colors other the colors used in $G'$ must be required to color $T_8$ by using the pigeon hole principle and accordingly derive the span.

\section{Results}\label{sec:2}

\begin{figure}
\begin{center}
\centerline{\includegraphics[scale=.85]{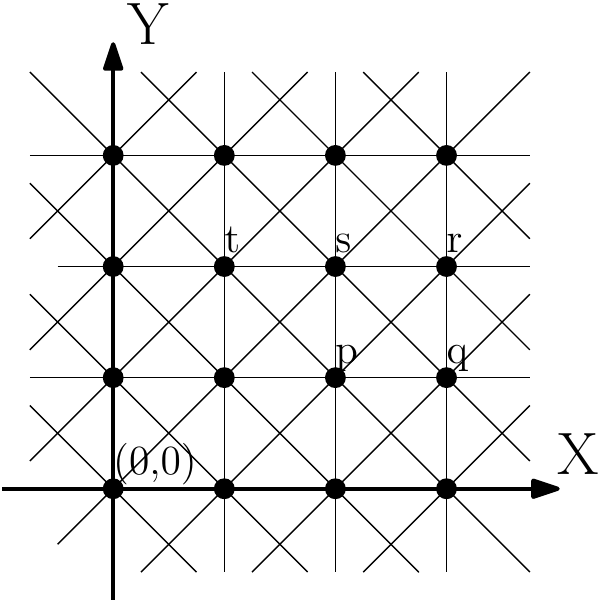}}
\caption{An infinite regular octagonal grid $T_8$.}
\label{oct_grid} 
\end{center} 
\end{figure}
Fig.~\ref{oct_grid} shows a portion of infinite regular octagonal grid $T_8$. Observe that there are four types of edges in $T_8$. The edges which are along or parallel to the $X$ axis and the edges which are along or parallel to the $Y$ axis are said to be \textit{horizontal} edges and \textit{vertical} edges respectively. The edges which are at $45^{\circ}$ to any of a horizontal edge and the edges which are at $135^{\circ}$ to any of a horizontal edge are said to be \textit{right slanting} edges and \textit{left slanting} edges respectively. Here the angular distance between two  adjacent edges $e_1$ and $e_2$ represents the smaller of the two angles measured in anticlockwise from $e_1$ to $e_2$ and from $e_2$ to $e_1$. In Fig.~\ref{oct_grid}, the edges $(p,q)$, $(p,r)$, $(p,s)$ and $(p,t)$ are a horizontal, a right slanting, a vertical and a slanting edges respectively. By \textit{slanting} edges, we mean all the left and right slanting edges and by  \textit{non slanting} edges, we mean all the horizontal and vertical edges.

\begin{figure}

\begin{center}
\centerline{\includegraphics[]{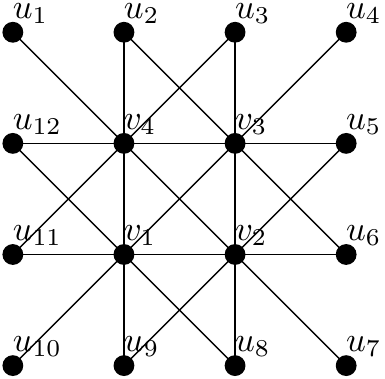}}
\caption{The $G_S$ corresponding to the $K_4$ having vertex set $S=\{v_1,v_2,v_3,v_4\}$.}
\label{oct_box} 
\end{center} 
\end{figure}

Consider any $K_4$ (complete graph of $4$ vertices) in $T_8$ and let $S=V(K_4)$ be the vertex set of the $K_4$. Let $v$ be a vertex in $T_8$ and $N_v$ be the set of vertices in $T_8$ which are adjacent to $v$. Let $N(S)=\displaystyle \bigcup_{v\in S} N_v$ be the set of all vertices in $T_8$ which are adjacent to at least one vertex in $S$. Let us define $G_S$ as the subgraph of $T_8$ such that $V(G_S)=S\cup N(S)$ and $E(G_S)$ is the set of all edges of $T_8$ which are incident to at least one vertex in $S$. Fig.~\ref{oct_box} shows the the $K_4$ with vertex set $S=\{v_1,v_2,v_3,v_4\}$ and the corresponding $G_S$. It is evident that any two edges of $G_S$ are at distance at most $2$. Hence no two edges of $G_S$ can be given the same color for $L(1,2)$-edge labeling of $G_S$. As $\vert E(G_{S}) \vert=26 $ and $d'(e_1,e_2) \leq 2$, $\forall e_1, e_2 \in E(G_{S})$, $\lambda'_{1,2}(G_{S})\geq 25$ for $L(1,2)$ edge labeling of $G_{S}$. 

\begin{figure}
\begin{center}
\centerline{\includegraphics[]{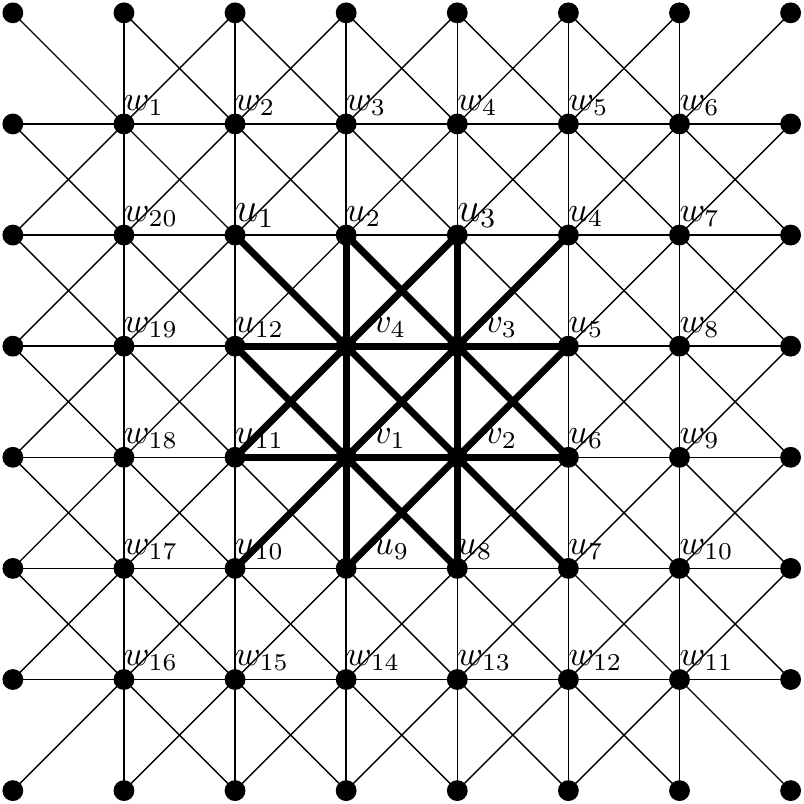}}
\caption{A subgraph $G$ of an infinite octagonal grid $T_8$.}
\label{oct} 
\end{center}
\end{figure}

Consider two edges $e_1, e_2 \in E(G_{S})$ such that $d'(e_1,e_2)=2$. Let the color $c$ be assigned to $e_1$. Clearly, the colors $c \pm  1$ can not be used at $e_2$ for $L(1,2)$-edge labeling. As there exists no pair of edges $e_1$ and $e_2$ such that $d'(e_1,e_2) \geq 3$ in $G_{S}$, both  $c \pm  1$ must be used at the adjacent edges of $e_1$ in $G_{S}$. 

Consider the $K_4$ having vertex set $S=\{v_1,v_2,v_3,v_4\}$ as shown in Fig. \ref{oct}. Now define $S'=S \cup N(S) \cup N(N(S))$. Fig. \ref{oct} shows the graph $G$ such that $V(G)=S' \cup N(S')$ and $E(G)$ is the set of all edges in $T_8$ which are incident to at least one vertex in $S'$. Note that there are total $25$ distinct $K_4$s including the $K_4$ having vertex set $S$ in $G$. Note that here $G$ is built over the $G_S$. Consider a $G_S$ and the corresponding $G$. Let two consecutive colors $c$ and $c+1$ be used in $G_S$. In the following lemma we identify all the $K_4$s having vertex sets $S_1, S_2, \cdots, S_m$ such that  $c$ and $c+1$ can not be used at $G_{S_1}, G_{S_2}, \cdots, G_{S_m}$ when $c$ and $c+1$ are used in $G_S$. 

\begin{lemma}\label{box2}
For every pair of consecutive colors $c$ and $c+1$ used in two adjacent edges $e_1$ and $e_2$ in $G_{S}$, except when $e_1$ and $e_2$ form  an angle of $45^{\circ}$ at their common incident vertex, there exists at least $4$ different $K_4$s having vertex sets $S_1$, $S_2$, $S_3$ and $S_4$ other than $S$ such that 1) $c$ can not be used at $G_{S_1}$ and $G_{S_2}$ and 2) $c+1$ can not be used at $G_{S_3}$ and $G_{S_4}$. When $e_1$ and $e_2$ are at angle $45^{\circ}$, there exists $3$ different $K_4$s having vertex sets $S_1$, $S_2$ and $S_3$ other than $S$ such that either 1) $c$ can not be used at $G_{S_1}$ and $G_{S_2}$ and $c+1$ can not be used at $G_{S_3}$ or 2) $c+1$ can not be used at $G_{S_1}$ and $G_{S_2}$ and $c$ can not be used at $G_{S_3}$.
\end{lemma}

\begin{proof}
Note that the angular distance between any two adjacent edges in $T_8$ can be any one of $45^{\circ}$, $90^{\circ}$, $135^{\circ}$, $180^{\circ}$, $225^{\circ}$, $270^{\circ}$ and $315^{\circ}$. From symmetry, it is suffice to consider the cases where the angular distance between two adjacent edges in $T_8$ is $180^{\circ}$ or $135^{\circ}$ or $90^{\circ}$ or $45^{\circ}$.

\begin{itemize}
    \item When angular distance between two adjacent edges is $180^{\circ}$: Observe that two adjacent horizontal edges or two adjacent vertical edges or two adjacent left slanting edges or two adjacent right slanting edges can be at $180^{\circ}$. Clearly the cases where two horizontal edges and  two vertical edges forming $180^{\circ}$ are symmetric. Similarly the cases where two left slanting edges and  two right slanting edges forming $180^{\circ}$ are also symmetric. So we need to consider the following two cases.\\

    Case $1$) When two adjacent horizontal edges $e_1=(v_1,v_2)$ and $e_2=(v_2, u_6)$ form $180^{\circ}$ (Fig. \ref{oct}). Let $f'(e_1)=c$ and $f'(e_2)=c+1$. Note that any edge $e$ incident to any of the vertices in $\{u_5,u_6,u_7\}$ is at distance at most two from $e_1$. As $f'(e_1)=c$, the color $c$ can not be used at any of those edges for $L(1,2)$ edge labeling. Similarly observe that any edge $e$ incident to any of the vertices in $\{w_8, w_9,w_{10}\}$ but not incident to any of the vertices in $\{u_5,u_6,u_7\}$ is at distance two from $e_2$. As $f'(e_2)=c+1$, the color $c$ can not be used at any of those edges for $L(1,2)$ edge labeling. Hence there exists $2$ different $K_4$s having vertex sets $S_1=\{ u_5,u_6,w_8,w_9\}$ and $S_2=\{u_6,u_7,w_9,w_{10} \}$ such that $c$ can not be used at $G_{S_1}$ and $G_{S_2}$. With similar argument, it can be shown that there exists $2$ different $K_4$s having vertex sets $S_3=\{v_1,v_4,u_{11},u_{12} \}$ and $S_4=\{ v_1,u_9, u_{10},u_{11} \}$ such that $c+1$ can not be used at $G_{S_3}$ and $G_{S_4}$.\\
    
    Case $2$) When two adjacent right slanting edges $e_1=(v_1,v_3)$ and $e_2=(v_1,u_{10})$ form $180^{\circ}$ (Fig. \ref{oct}). Let $f'(e_1)=c$ and $f'(e_2)=c+1$. In this case, 
there exists $3$ different $K_4$s having vertex sets $S_1=\{u_{10},w_{15}, w_{16}, w_{17}\}$,  $S_2=\{u_9, u_{10},w_{14},w_{15}\}$ and $S_3=\{u_{10},u_{11},w_{17},w_{18}\}$ such that $c$ can not be used at $G_{S_1}$,  $G_{S_2}$ and  $G_{S_3}$. Similarly, there exists $3$ different $K_4$s having vertex sets $S_3=\{v_3,v_4,u_2,u_3 \}$,  $S_4=\{ v_3,u_3,u_4, ,u_5\}$ and $S_5=\{ v_2,v_3,u_5, u_6\}$  such that $c+1$ can not be used at $G_{S_3}$, $G_{S_4}$ and $G_{S_5}$.\\

\item When angular distance between two adjacent edges is $135^{\circ}$: Observe that a horizontal edge and its adjacent left slanting edge or a horizontal edge and its adjacent right slanting edge or a vertical edge and its adjacent left slanting edge or a vertical edge and its adjacent right slanting edge may be at $135^{\circ}$. Clearly all the cases are symmetric. So we need to consider the following case only.  Let us consider the horizontal edge $e_1=(v_1,v_2)$ and the right slanting edge $e_2=(v_2,u_5)$. Let  $f'(e_1)=c$ and $f'(e_2)=c+1$. From similar discussion stated in the previous case, there exists $3$ different $K_4$s having vertex sets $S_1=\{v_3,u_3 ,u_4, u_5\}$,  $S_2=\{u_4,u_5,w_7,w_8 \}$ and $S_3=\{u_5,u_6, w_8,w_9 \}$ such that $c$ can not be used at $G_{S_1}$, $G_{S_2}$ and $G_{S_3}$. Similarly, there exists $2$ different $K_4$s having vertex sets $S_4=\{v_1,v_4,u_{11},u_{12} \}$ and $S_5=\{ v_1,u_9, u_{10},u_{11}\}$ such that $c+1$ can not be used at $G_{S_4}$ and $G_{S_5}$.\\

\item When angular distance between two adjacent edges is $90^{\circ}$: Observe that a horizontal edge and its adjacent vertical edge or a vertical edge and its adjacent horizontal edge or a right slanting edge and its adjacent left slanting edge  may be at $90^{\circ}$. Clearly the first two cases are symmetric. So we need to consider the following two cases.\\

Case $1$) Consider the horizontal edge $e_1=(v_1,v_2)$ and the vertical edge $e_2=(v_2,v_3)$. Let $f'(e_1)=c$ and $f'(e_2)=c+1$. As similar argument stated in the previous cases, 
there exists $2$ different $K_4$s having vertex sets $S_1=\{v_3,v_4,u_2,u_3 \}$ and $S_2=\{v_3,u_3,u_4,u_5\}$ such that $c$ can not be used at $G_{S_1}$ and $G_{S_2}$. Similarly, there exists $2$ different $K_4$s having vertex sets $S_3=\{v_1,v_4,u_{11},u_{12}\}$ and $S_4=\{ v_1,u_9,u_{10},u_{11} \}$ such that $c+1$ can not be used at $G_{S_3}$ and $G_{S_4}$.\\

Case $2$) Consider the right slanting edge $e_1=(v_1,v_3)$ and the left slanting edge $e_2=(v_1,u_{12})$. Let $f'(e_1)=c$ and $f'(e_2)=c+1$. As similar argument stated in the previous cases, 
there exists $3$ different $K_4$s having vertex sets $S_1=\{u_1,u_{12},w_{19},w_{20}\}$,  $S_2=\{u_1,u_2,u_{12},v_4\}$ and $S_3=\{u_{11},u_{12},w_{18},w_{19}\}$ such that $c$ can not be used at $G_{S_1}$, $G_{S_2}$ and $G_{S_3}$. Similarly, there exists $3$ different $K_4$s having vertex sets $S_4=\{v_3,v_4,u_2,u_3 \}$, $S_5=\{ v_3,u_3,u_4,u_5 \}$ and $S_6=\{ v_2,v_3,u_5, u_6\}$ such that $c+1$ can not be used at $G_{S_4}$,  $G_{S_5}$ and $G_{S_6}$.\\

\item When angular distance between two edges is $45^{\circ}$:  
Observe that a horizontal edge and its adjacent left slanting edge or a horizontal edge and its adjacent right slanting edge or a vertical edge and its adjacent left slanting edge or a vertical edge and its adjacent right slanting edge may be at $45^{\circ}$. Clearly all the cases are symmetric. So we need to consider the following case only. Consider the horizontal edge $e_1=(v_1,v_2)$ and the left  slanting edge $e_2=(v_2,v_4)$. Let  $f'(e_1)=c$ and $f'(e_2)=c+1$. As similar argument stated in the previous cases, there exists $2$ different $K_4$s having vertex sets $S_1=\{v_4,u_1, u_2,u_{12} \}$ and $S_2=\{v_3,v_4,u_2,u_3 \}$  such that $c$ can not be used at $G_{S_1}$ and $G_{S_2}$. Similarly, there exists a $K_4$s having vertex set $S_3=\{ v_1,u_9,u_{10}, u_{11}\}$ such that $c+1$ can not be used at $G_{S_3}$.  If $f'(e_1)=c+1$ and $f'(e_2)=c$, then there exists $2$ different $K_4$s having vertex sets $S_1=\{v_4,u_1, u_2,u_{12} \}$ and $S_2=\{v_3,v_4,u_2,u_3 \}$  such that $c+1$ can not be used at $G_{S_1}$ and $G_{S_2}$. Similarly, there exists a $K_4$s having vertex set $S_3=\{ v_1,u_9,u_{10},u_{11}\}$ such that $c$ can not be used at $G_{S_3}$.  

\end{itemize}
\end{proof} \qed

Now we state and prove the following Theorems.
\begin{theorem}
$\lambda'_{1,2}(T_8)\geq 26$.
\end{theorem}
\begin{proof}
Let us consider  the $K_4$ with vertex set $S=\{v_1,v_2,v_3,v_4\}$ and the corresponding $G_{S}$ as shown in Fig. \ref{oct_box}. Observe that there are $26$ edges in $G_{S}$. Note that there are no two edges at more than distance two apart in $G_{S}$. Hence all the colors used in $G_{S}$ must be distinct for $L(1,2)$-edge labeling. Hence $26$ consecutive colors $\{0,1, \cdots, 25\}$ are to be used in $G_{S}$, otherwise $\lambda'_{1,2}(G_{S})\geq 26$.  In that case, any two consecutive colors $c$ and $c+1$ must be  used at two adjacent edges in $G_{S}$. Therefore, from Lemma~\ref{box2}, there exists at least a  $K_4$ having vertex set $S_1$ such that $c$ can not be used at $G_{S_1}$. But in $G_{S_1}$, $26$ distinct colors must be used. So if we do not use $c$ in $G_{S_1}$, at least a new color which is not used in $G_{S}$ must be used in $G_{S_1}$. So at least the color $26$ must be introduced in $G_{S_1}$. Hence $\lambda'_{1,2}(G_{S_1})\geq 26$ implying $\lambda'_{1,2}(T_8)\geq 26$. 
\end{proof}\qed

\begin{theorem}\label{th2}
$\lambda'_{1,2}(T_8)\geq 27$.
\end{theorem}

\begin{proof}
Consider  $S=\{v_1,v_2,v_3,v_4\}$ and the corresponding $G_{S}$ as shown in Fig. \ref{oct}. Note that $\lambda'_{1,2}(G_S) \leq 26$ only if 1) $26$ consecutive colors $\{0, 1, \cdots, 25\}$ or $\{1, 2, \cdots, 26\}$ are used in $G_{S}$, or 2) the colors $\{0,1,\cdots,26 \}\setminus \{c'\}$ are used in $G_{S}$, where $ 1 \leq c' \leq 25$. The cases of $\{0, 1, \cdots, 25\}$ and $\{1, 2, \cdots, 26\}$ are clearly symmetric. Hence we need to consider only the following  two cases.

Consider the first case. Observe that the set of colors $\{0, 1, \cdots, 25\}$ can be partitioned into $13$ disjoint consecutive pairs of colors $(0,1)$, $(2,3)$, $\cdots $, $(24,25)$. Consider a pair of consecutive colors $(c,c+1)$, where $0 \leq c \leq 24$. From  Lemma~\ref{box2}, it follows that for the pair $(c,c+1)$, there exists two distinct $S_1$ and $S_2$ other that $S$ such that either $c$ or $c+1$ can not be used in $G_{S_1}$ and $G_{S_2}$. So for the $13$ disjoint pairs mentioned above,  there must be $26$ such $S_1, S_2, \cdots, S_{26}$ other than $S$. Note that there are total $25$ $S_j$s including $S$ in $G$. Therefore, there exists only $24$ such $S_j$s other than $S$ in $G$. Hence from pigeon hole ($26$ pigeons and $24$ holes) principle there must be at least one $G_{S_j}$ where two different colors which are used in $G_S$ can not be used there and hence $\lambda'_{1,2}(G_{S_j})\geq 27$ implying $\lambda'_{1,2}(T_8)\geq 27$.

Now we consider the second case. Consider that the colors $\{0,1,\cdots,26 \}\setminus \{c'\}$ have been used in $G_{S}$, where $ 1 \leq c' \leq 25$. First let us consider the case when $c'$ is even. In that case the set of colors $\{0, \cdots, c'-1, c'+1, \cdots 26\}$ can be partitioned into $13$ disjoint consecutive pairs of colors $(0,1), \cdots, (c'-2,c'-1), (c'+1,c'+2), \cdots, (25,26)$. Hence proceeding similarly as above case, from pigeon hole principle, we get that $\lambda'_{1,2}(T_8)\geq 27$. Now consider the case when $c'$ is odd. Let us first consider $c' \neq 1,25$. Note that the set of colors $\{0, \cdots, c'-1, c'+1, \cdots 26\}$ can be partitioned into $12$ disjoint consecutive pairs of colors $(0,1),\cdots, (c'-3,c'-2),(c'+1,c'+2),\cdots, (24,25)$. So there must be $24$ $K_4$s having vertex sets $S_1, S_2, \cdots, S_{24}$ other than $S$ in $G$. Now consider the pair of consecutive colors $(25,26)$. Now, from Lemma \ref{box2}, there must exists at least one $S_{25}$ other than $S$ such that color $26$ can not be used in $G_{S_{25}}$. So we need $25$ such $S_1, S_2, \cdots, S_{25}$ other than $S$. But there exists only $24$ such distinct $S_j$s other than $S$ in $G$. Hence from pigeon hole ($25$ pigeons and $24$ holes) principle there must be at least one $G_{S_j}$ where two different colors used in $G_S$ can not be used there and hence $\lambda'_{1,2}(G_{S_j})\geq 27$ implying $\lambda'_{1,2}(T_8)\geq 27$.
 When $c'=1$ the set of colors $\{0, 2, \cdots 26\}$  can be partitioned into $12$ disjoint consecutive pairs of colors  $(2,3),\cdots, (24,25)$. Considering these $12$ pairs and the pair of consecutive colors $(25,26)$, we get $\lambda'_{1,2}(T_8)\geq 27$ by proceeding similarly as above.
When $c'=25$ the set of colors $\{0,  \cdots 24,26\}$  can be partitioned into $12$ disjoint consecutive pairs of colors  $(1,2,), \cdots, (23,24)$. Considering these $12$ pairs and the pair of consecutive colors $(0,1)$, we get $\lambda'_{1,2}(T_8)\geq 27$ by proceeding similarly as above.
\end{proof}\qed

 Observe that a $K_3$ contains one horizontal, one vertical and one slanting edge. Note that three consecutive colors $c-1,c,c+1$ can be used at three edges with or without forming a $K_3$ (complete graph of $3$ vertices) in $G_S$, where $c$ being used at a slanting or non slanting edge. Accordingly, we now have the following two Lemmas.

\begin{figure}
\begin{center}
\centerline{\includegraphics[scale=.65]{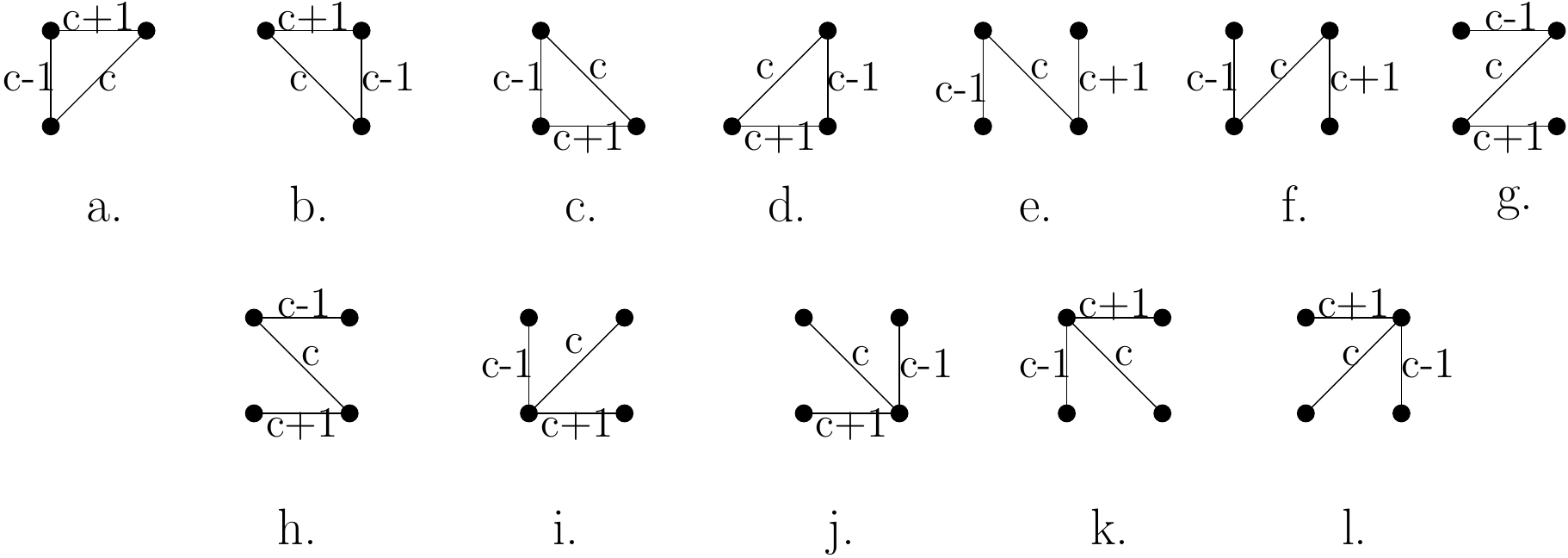}}
\caption{The all possible cases where $c$ is used in a slanting edge $e$ and both $c \pm 1$ are used at edges forming $45^{\circ}$ with $e$.}
\label{dis45} 
\end{center} 
\end{figure}

\begin{lemma}\label{oneuse}
 When three consecutive colors $c-1$, $c$ and $c+1$ are used at  three edges {\it forming} a $K_3$ in $G_S$ with $c$ being used at the {\it slanting edge}, then there exists a  $K_4$ with vertex set $S_1$ in $G$  such that $c$ can not be used in $G_{S_1}$. When three consecutive colors $c-1$, $c$ and $c+1$ are used at three edges {\it without forming} a $K_3$ in $G_S$ with $c$ being used at the {\it slanting edge}, then there exists at least $2$ different $K_4$s with vertex sets $S_1$ and $S_2$ in $G$ such that $c$ can not be used at  $G_{S_1}$ and $G_{S_2}$. 
\end{lemma}

\begin{proof}
Let us consider $c-1$, $c$ and $c+1$ are used at three edges which are forming a $K_3$ with $c$ being used at the slanting edge. The different cases where such scenarios occur are shown in Fig.~\ref{dis45} a. to Fig.~\ref{dis45} d. Note that all the cases shown in Fig.~\ref{dis45} a. to Fig.~\ref{dis45} d. are symmetric. So we consider any one of the cases. Consider the case where the colors $c-1$, $c$ and $c+1$ are used at three edges $(v_1,v_2)$, $(v_1,v_3)$ and $(v_2,v_3)$ respectively (Fig.~\ref{oct}). Now consider the $K_4$ with vertex set $S_1=\{v_2,u_6,u_7,u_8\}$ and observe that any edge incident to any of the vertices in $S_1 \setminus \{v_2\}$ is  at distance two from $(v_1,v_2)$. As $f'(v_1,v_2)=c-1$, $c$ can not be used at those edges for $L(1,2)$ edge labeling. As any edge incident to $v_2$ is at distance at most $2$ from $(v_1,v_3)$ and $f'(v_1,v_3)=c$, the color $c$ can not be used there as well. So $c$ can not be used at $G_{S_1}$. 

Now we consider the case where $c-1$, $c$ and $c+1$ are used at three edges which are not forming a $K_3$ with $c$ being used at the slanting edge. Note that the colors $c-1$ and $c+1$ can be used at two edges both of which are at $45^{\circ}$ with the edge having color $c$. The different cases where such scenarios arrive are shown in Fig.~\ref{dis45} e. to Fig.~\ref{dis45} l. Note that the cases shown in Fig.~\ref{dis45} e. to Fig.~\ref{dis45} h. are symmetric and the cases shown in Fig.~\ref{dis45} i. to Fig.~\ref{dis45} l. are also symmetric. So here we need to consider only two cases.

\begin{itemize}
    
\item Consider the case where the colors $c-1$, $c$ and $c+1$ are used at three edges $(v_1,v_4)$, $(v_1,v_3)$ and $(v_2,v_3)$ respectively (Fig.~\ref{oct}). In this case, the said three edges form a structure isomorphic to Fig.\ref{dis45} e. to Fig.\ref{dis45} h. Here note that there exists two $K_4$s with vertex sets  $S_1=\{v_4,u_1,u_2,u_{12}\}$ and $S_2=\{v_2,u_6,u_7,u_8\}$ such that $c$ can not be used in $G_{S_1}$ and $G_{S_2}$.

\item Consider the case where the colors $c-1$, $c$ and $c+1$ are used at three edges $(v_1,v_4)$, $(v_1,v_3)$ and $(v_1,v_2)$ respectively (Fig.~\ref{oct}). In this case the said three edges form a structure isomorphic to Fig. \ref{dis45} i. to Fig. \ref{dis45} l. Here also, there exists two $K_4$s with vertex sets $S_1=\{v_4,u_1,u_2,u_{12}\}$ and $S_2=\{v_2,u_6,u_7,u_8\}$ such that  $c$ can not in  $G_{S_1}$ and $G_{S_2}$.
\end{itemize}

Now we consider the case where the edge having color $c+1$ (or $c-1$) is not forming $45^{\circ}$ with the slanting edge having color $c$ and the other edge having color $c-1$ (or $c+1$) forming  $45^{\circ}$ with the slanting edge. In that case, from Lemma~\ref{box2}, it follows that there exists at least $2$ distinct $K_4$s having vertex sets $S_1$ and $S_2$ other than $S$ such that $c$ can not be used in $G_{S_1}$ and $G_{S_2}$. As the other color $c-1$ (or $c+1$) is used at an edge forming $45^{\circ}$ with the slanting edge, there also exists at least another  $K_4$ having vertex set $S_3$ other than $S$ such that $c$ can not be used in $G_{S_3}$. So in this case there are at least $3$ distinct $K_4$s. 

We now consider the case where both the colors $c-1$ and $c+1$ both are used at edges not forming $45^{\circ}$ with the slanting edge with color $c$. In this case, from Lemma~\ref{box2}, there exists at least $4$ different $K_4$s having vertex sets $S_1$, $S_2$, $S_3$ and $S_4$ other than $S$ such that $c$ can not be used in $G_{S_1}$, $G_{S_2}$, $G_{S_3}$ and $G_{S_4}$. So in this case there are at least $4$ distinct $K_4$s. 

\end{proof}\qed

\begin{lemma}\label{twouse}

If three consecutive colors $c-1$, $c$ and $c+1$ are used at  three edges of $G_S$ {\it forming} a $K_3$ with $c$ being used at a {\it non slanting edge}, then there exists two  $K_4$s with vertex sets $S_1$ and $S_2$ in $G$  such that $c$ can not be used in $G_{S_1}$ and $G_{S_2}$.  If three consecutive colors $c-1$, $c$ and $c+1$ are used at  three edges of $G_S$ {\it without forming} a $K_3$ with $c$ being used at a {\it non slanting edge}, then there exists at least $3$ $K_4$s with vertex sets $S_1$,  $S_2$ and $S_3$ in $G$  such that $c$ can not be used in $G_{S_1}$, $G_{S_2}$ and $G_{S_2}$. 
\end{lemma}
\begin{proof}
\begin{itemize}
\item Let us first consider the case where $c-1$, $c$ and $c+1$ are used at three edges which are forming a $K_3$ with $c$ being used at the non slanting edge. Consider the $K_3$ with vertex set $\{ v_1, v_2, v_3\}$ in $G_{S}$ as shown in Fig.~\ref{oct}. Suppose the colors $c-1$, $c$ and $c+1$ are used at  $(v_1,v_2)$, $(v_2,v_3)$ and $(v_1,v_3)$ respectively. In this case, there exists  two $K_4$s with vertex sets  $S_1=\{v_1,u_9,u_{10},u_{11}\}$ and $S_2=\{v_1,v_4,u_{11},u_{12} \}$ such that $c$ can not be used at $G_{S_1}$ and $G_{S_2}$. 

\item Now we consider the case where $c-1$, $c$ and $c+1$ are used at three edges which are not forming a $K_3$ with $c$ being used at the non slanting edge. In the previous case, as the two edges with colors $c-1$ and $c+1$ have a common vertex other than the end vertices of the edge with color $c$, we get only two different $K_4$s. Note that if three edges do not form a $K_3$, there does not exists a common vertex of the two edges with colors $c-1$ and $c+1$, other than an end vertex of the edge with color $c$. In that case, there exists at least $3$ different $K_4$s having vertex sets $S_1$, $S_2$ and $S_3$ such that $c$ can not be used at $G_{S_1}$, $G_{S_2}$ and $G_{S_3}$. 
\end{itemize}
\end{proof} \qed

We now consider how many $K_3$S can be there in $G_S$ such that for each such $K_3$ three consecutive colors can be used. For this, we have the following Lemma.
\begin{observation}\label{eight}
Consider the $K_4$ having vertex set $S=\{v_1,v_2,v_3,v_4\}$ and the subgraph $G_S$. There can be at most $8$ $K_3$s in $G_S$ such that in each of the $K_3$, three consecutive colors can be used. 
\end{observation}

\begin{proof}
Consider the edges $e_1=(v_1,v_2)$, $e_2=(v_2,v_3)$, $e_3=(v_3,v_4)$ and $e_4=(v_4,v_1)$ of the $K_4$ having vertex set $S=\{v_1,v_2,v_3,v_4\}$ (Fig. \ref{oct_box}). Observe that there are total $12$ $K_3$s in $G_S$ and each $K_3$ has at least one edge in $\{e_1,e_2,e_3,e_4\}$. Each edge $e \in \{e_1,e_2,e_3,e_4\}$ is a common edge of $4$ different $K_3$s. Let a color $c$ be used in an edge $e\in \{e_1,e_2,e_3,e_4\}$. In order to form a $K_3$ with three consecutive colors with $e$, either $c+1$ or $c-1$ must be used in that $K_3$.  So, out of the $4$ $K_3$s that include $e$, at most two of then can have $3$ consecutive colors. As there are $4$ edges in  $\{e_1,e_2,e_3,e_4\}$, we can have at most $8$ $K_3$s where each of them has $3$ consecutive colors.
\end{proof} \qed

Now we state and prove the following Theorem. 
\begin{theorem}\label{th3}
$\lambda'_{1,2}(T_8)\geq 28$.
\end{theorem}
\begin{proof}
Consider the graph $G$, the $K_4$ with vertex set $S=\{v_1,v_2,v_3,v_4\}$ and the subgraph $G_S$ as shown in Fig. \ref{oct}. Note that $26$ distinct colors from $\{0, 1, \cdots, 27\}$ must be used in $G_{S}$. In other words, there must be $2$ colors in $\{0, 1, \cdots, 27\}$ which should remain unused. Let these two colors be $c_1$ and $c_2$ where $c_1 < c_2$. Let us now consider the $6$ colors $0$, $c_1-1$, $c_1+1$, $c_2-1$, $c_2+1$ and $27$. 

We first consider the case when all these $6$ colors are distinct and denote $X=\{0, c_1-1, c_1+1, c_2-1, c_2+1 27\}$. In this case, for each color $c \in X$, only one of $c \pm 1$ is used and the other is not used in $G_S$. From Lemma \ref{box2}, for each $c \in X$,  there exists at least one $K_4$ having vertex set $S_1$  other than $S$ in $G$ such that $c$ can not be used at $G_{S_1}$. Moreover, $c$ must be used at a slanting edge in $G_S$ in this case. Considering all $6$ colors in $X$ are used in $6$ slanting edges, we get at least $6$ such $K_4$s. For each $c$ among the remaining $26-6=20$ colors, both $c\pm 1$ are used in $G_S$. There are total $26$ edges in $G_S$ among which $14$ are slanting edges and $12$ are non-slanting edges. Therefore, we are yet to consider the colors used in the remaining $14-6=8$ slanting edges. Note that for each such color $c$, both $c\pm 1$ are used in $G_S$. Assume these $8$ colors, $x$ many colors $c'_1, c'_2, \cdots, c'_x$ are there such that for each $c'_i$, three consecutive colors $c'_i-1$, $c'_i$, $c'_i + 1$ can be used in a $K_3$. From Lemma \ref{oneuse}, for each such color $c'_i$, there exists at least one $K_4$ having vertex set $S_1$  other than $S$ in $G$ such that $c'_i$ can not be used at $G_{S_1}$. Considering all these $x$ colors, we get at least $x$ such $K_4$s in $G$. Now consider the remaining  $8-x$ colors $c'_{x+1}, c'_{x+2}, \cdots, c'_8$ used in slanting edges. Note that for each such $c'_i$, three consecutive colors $c'_i-1$, $c'_i$ and $c'_i+1$ can not be used in a $K_3$. From Lemma \ref{oneuse}, for each such $c'_i$,  there exists at least $2$ different $K_4$s having vertex sets $S_1$ and  $S_2$  other than $S$ in $G$ such that $c'_i$ can not be used at $G_{S_1}$ and $G_{S_2}$. Considering all these $8-x$ colors, we get at least $2(8-x)$ $K_4$s in $G$. 

Now consider the $12$ non slanting edges in $G_S$.  Assume that among them, $y$ many colors $c''_1, c''_2, \cdots, c''_y$ are there such that for each  $c''_i$, three consecutive colors $c''_i-1$, $c''_i$, $c''_i + 1$ can be used in a $K_3$. Clearly all those $y$ many $K_3$s must be different from those $x$ many $K_3$ considered for slanting edges. From Observation \ref{eight}, there exists at most $8$ $K_3$s in $G_S$ such that for each of them, three consecutive colors can be used. As $x$ many $K_3$s have already been considered for slanting edges, $y$ can be at most $8-x$. From Lemma\ref{twouse}, for each such $c''_i$,  there exists at least $2$ different $K_4$s having vertex sets $S_1$ and  $S_2$  other than $S$  such that $c''_i$ can not be used at $G_{S_1}$ and  $G_{S_2}$. Considering all these $8-x$ colors, we get at least $2(8-x)$  $K_4$s. We are yet to consider the remaining $z=12-(8-x)=4+x$ non slanting edges. For each such color $c''$, the colors $c''-1$, $c''$ and $c''+1$ can not be used in a $K_3$ in $G_S$. So from Lemma\ref{twouse}, for each such $c''$,  there exists at least $3$ different $K_4$s having vertex sets $S_1$, $S_2$ and $S_3$ other than $S$ in $G$ such that $c''$ can not be used at $G_{S_1}$, $G_{S_2}$ and  $G_{S_3}$. Considering all these $4+x$ colors, we get at least $3(4+x)$ $K_4$s. In total we get at least $6+x+2(8-x)+2(8-x)+3(4+x)=50$ $K_4$s in $G$. But there are only $24$ distinct $K_4$s in $G$ other than $S$. From pigeon hole principle ($50$ pigeons and $24$ holes), there exists at least one $S_i$ in $G$ such that at least $3$ colors which are used in $G_S$ can not be used in $G_{S_i}$ and hence $\lambda'_{1,2}(G_{S_i})\geq 28$ implying $\lambda'_{1,2}(T_8)\geq 28$. 

If the colors $0$, $c_1-1$, $c_1+1$, $c_2-1$, $c_2+1$ and $27$ are not distinct, proceeding similarly, we can show that there are a need of more that $50$ $K_4$s in $G$ and hence from pigeon hole principle (more than $50$ pigeons and $24$ holes) $\lambda'_{1,2}(T_8)\geq 28$.

\end{proof}\qed

\section{Conclusion}
It was proved in~\cite{cala} that $25 \leq \lambda'_{1,2}(T_8)\leq 28$ with a gap between the lower and upper bounds. In this paper, we filled the gap and proved that $\lambda'_{1,2}(T_8)\geq 28$. This essentially implies $\lambda'_{1,2}(T_8)= 28$.

\bibliographystyle{splncs04}
\bibliography{mybibfile}

\end{document}